\documentclass{amsart}
\usepackage{amssymb,amsmath,amsthm}
\usepackage[dvips]{graphicx}

\let\oldmarginpar\marginpar
\renewcommand\marginpar[1]{\-\oldmarginpar[\raggedleft\footnotesize #1]%
{\raggedright\footnotesize #1}}

\begin{document}

\newtheorem{theorem}{Theorem}[section]
\newtheorem{corollary}[theorem]{Corollary}
\newtheorem{lemma}[theorem]{Lemma}
\newtheorem{proposition}[theorem]{Proposition}
\theoremstyle{definition}
\newtheorem{definition}[theorem]{Definition}
\theoremstyle{remark}
\newtheorem{remark}[theorem]{Remark}
\theoremstyle{definition}
\newtheorem{example}[theorem]{Example}

\def\R{{\mathbb R}}
\def\H{{\mathbb H}}
\def\rank{{\text{rank}\,}}
\def\bd{{\partial}}

\title{Semi-slant submersions}

\author{Kwang-Soon Park}
\address{Department of Mathematical Sciences, Seoul National University, Seoul 151-747, Republic of Korea}
\email{parkksn@gmail.com}

\author{Rajendra Prasad }
\address{Department of Mathematics and Astronomy, University of Lucknow, Lucknow-226007, India}
\email{rp.manpur@rediffmail.com}






\begin{abstract}
Let $F$ be a Riemannian submersion from an almost Hermitian
manifold $(M,g_M,J)$ onto a Riemannian manifold $(N,g_N)$. We
introduce the notion of the semi-slant submersion. And then we
obtain some properties on it. In particular, we give some examples
for it.
\end{abstract}

\maketitle
\section{Introduction}\label{intro}
\addcontentsline{toc}{section}{Introduction}

Let $F$ be a $C^{\infty}$-submersion from a semi-Riemannian
manifold $(M,g_M)$ onto a semi-Riemannian manifold $(N,g_N)$. Then
according to the conditions on the map $F : (M,g_M) \mapsto
(N,g_N)$, we have the following submersions:

semi-Riemannian submersion and Lorentzian submersion \cite{FIP},
Riemannian submersion (\cite{G}, \cite{O}), slant submersion
(\cite{C}, \cite{S}), almost Hermitian submersion \cite{W},
contact-complex submersion \cite{IIMV}, quaternionic submersion
\cite{IMV}, almost h-slant submersion and h-slant submersion
\cite{P}, semi-invariant submersion \cite{S2}, h-semi-invariant
submersion \cite{P2}, etc. As we know, Riemannian submersions are
related with physics and have their applications in the Yang-Mills
theory (\cite{BL2}, \cite{W2}), Kaluza-Klein theory (\cite{BL},
\cite{IV}), Supergravity and superstring theories (\cite{IV2},
\cite{M}), etc. Let $(M, g_M)$ and $(N, g_N)$ be Riemannian
manifolds and $F : M \mapsto N$ a $C^{\infty}-$submersion. The map
$F$ is said to be {\em Riemannian submersion} if the differential
$F_*$ preserves the lengths of horizontal vectors \cite{IMV}. Let
$(M, g_M, J)$ be an almost Hermitian manifold. A Riemannian
submersion $F : (M, g_M, J)\mapsto (N, g_N)$ is called a {\em
slant submersion} if the angle $\theta(X)$ between $JX$ and the
space $\ker (F_*)_p$ is constant for any nonzero $X\in T_p M$ and
$p\in M$ \cite{S}. We call $\theta(X)$ a {\em slant angle}. A
Riemannian submersion $F : (M, g_M, J)\mapsto (N, g_N)$ is called
a {\em semi-invariant submersion} if there is a distribution
$\mathcal{D}_1 \subset \ker F_*$ such that
$$
\ker F_*=\mathcal{D}_1\oplus \mathcal{D}_2, \
 J(\mathcal{D}_1)=\mathcal{D}_1, \ J(\mathcal{D}_2)\subset (\ker
F_*)^{\perp},
$$
where $\mathcal{D}_2$ is the orthogonal complement of
$\mathcal{D}_1$ in $\ker F_*$ \cite{S}. Let $(M, g_M)$ and $(N,
g_N)$ be Riemannian manifolds and $F : (M, g_M) \mapsto (N, g_N)$
a smooth map. The second fundamental form of $F$ is given by
$$
(\nabla F_*)(X,Y) := \nabla^F _X F_* Y-F_* (\nabla _XY) \quad
\text{for} \ X,Y\in \Gamma(TM),
$$
where $\nabla^F$ is the pullback connection and we denote
conveniently by $\nabla$ the Levi-Civita connections of the
metrics $g_M$ and $g_N$ [4]. Recall that $F$ is said to be {\em
harmonic} if $trace (\nabla F_*)=0$ and $F$ is called a {\em
totally geodesic} map if $(\nabla F_*)(X,Y)=0$ for $X,Y\in \Gamma
(TM)$ \cite{BW}. The paper is organized as follows. In section 2
we give the definition of the semi-slant submersion and obtain
some interesting properties on them. In section 3 we construct
some examples for the semi-slant submersion.

\section{Semi-slant submersions}\label{semi}

\begin{definition}
Let $(M,g_M,J)$ be an almost Hermitian manifold and $(N,g_N)$ a
Riemannian manifold. A Riemannian submersion $F : (M,g_M,J)\mapsto
(N,g_N)$ is called a {\em semi-slant submersion} if there is a
distribution $\mathcal{D}_1\subset \ker F_*$ such that
$$
\ker F_* =\mathcal{D}_1\oplus \mathcal{D}_2, \
J(\mathcal{D}_1)=\mathcal{D}_1,
$$
and the angle $\theta=\theta(X)$ between $JX$ and the space
$(\mathcal{D}_2)_q$ is constant for nonzero $X\in
(\mathcal{D}_2)_q$ and $q\in M$, where $\mathcal{D}_2$ is the
orthogonal complement of $\mathcal{D}_1$ in $\ker F_*$.
\end{definition}

We call the angle $\theta$ a {\em semi-slant angle}.

\noindent Let $F : (M,g_M,J)\mapsto (N,g_N)$ be a semi-slant
submersion. Then there is a distribution $\mathcal{D}_1\subset
\ker F_*$ such that
$$
\ker F_* =\mathcal{D}_1\oplus \mathcal{D}_2, \
J(\mathcal{D}_1)=\mathcal{D}_1,
$$
and the angle $\theta=\theta(X)$ between $JX$ and the space
$(\mathcal{D}_2)_q$ is constant for nonzero $X\in
(\mathcal{D}_2)_q$ and $q\in M$, where $\mathcal{D}_2$ is the
orthogonal complement of $\mathcal{D}_1$ in $\ker F_*$.

Then for $X\in \Gamma(\ker F_*)$, we have
$$
X = PX+QX,
$$
where $PX\in \Gamma(\mathcal{D}_1)$ and $QX\in
\Gamma(\mathcal{D}_2)$.

For $X\in \Gamma(\ker F_*)$, we get
$$
JX = \phi X+\omega X,
$$
where $\phi X\in \Gamma(\ker F_*)$ and $\omega X\in \Gamma((\ker
F_*)^{\perp})$.

For $Z\in \Gamma((\ker F_*)^{\perp})$, we obtain
$$
JZ = BZ+CZ,
$$
where $BZ\in \Gamma(\ker F_*)$ and $CZ\in \Gamma((\ker
F_*)^{\perp})$.

For $U\in \Gamma(TM)$, we have
$$
U = \mathcal{V}U+\mathcal{H}U,
$$
where $\mathcal{V}U\in \Gamma(\ker F_*)$ and $\mathcal{H}U\in
\Gamma((\ker F_*)^{\perp})$.

Then
$$
(\ker F_*)^{\perp} = \omega \mathcal{D}_2 \oplus \mu,
$$
where $\mu$ is the orthogonal complement of $\omega \mathcal{D}_2$
in $(\ker F_*)^{\perp}$ and is invariant  under $J$. Furthermore,

{\setlength\arraycolsep{2pt}
\begin{eqnarray*}
& & \phi \mathcal{D}_1 = \mathcal{D}_1, \omega \mathcal{D}_1 = 0, \phi \mathcal{D}_2 \subset \mathcal{D}_2,
B((\ker F_*)^{\perp}) = \mathcal{D}_2        \\
& & \phi^2+B\omega = -id, C^2+\omega B = -id, \omega \phi +C\omega
= 0, BC+\phi B = 0.
\end{eqnarray*}}

Define the tensors $\mathcal{T}$ and $\mathcal{A}$ by
\begin{align*}
  &\mathcal{A}_E F= \mathcal{H}\nabla_{\mathcal{H}E} \mathcal{V}F+\mathcal{V}\nabla_{\mathcal{H}E} \mathcal{H}F \\
   &\mathcal{T}_E F= \mathcal{H}\nabla_{\mathcal{V}E} \mathcal{V}F+\mathcal{V}\nabla_{\mathcal{V}E}
   \mathcal{H}F
\end{align*}
for vector fields $E, F$ on $M$, where $\nabla$ is the Levi-Civita
connection of $g_M$. Define
$$
(\nabla_X \phi)Y := \widehat{\nabla}_X \phi Y-\phi
\widehat{\nabla}_X Y
$$
and
$$
(\nabla_X \omega)Y := \mathcal{H}\nabla_X \omega
Y-\omega\widehat{\nabla}_X Y
$$
for $X,Y\in \Gamma(\ker F_*)$, where $\widehat{\nabla}_X Y :=
\mathcal{V}\nabla_X Y$. Then we easily have

\begin{lemma}
Let $(M,g_M,J)$ be a K\"{a}hler manifold and $(N,g_N)$ a
Riemannian manifold. Let $F : (M,g_M,J) \mapsto (N,g_N)$ be a
semi-slant submersion. Then we get
\begin{enumerate}
\item
\begin{align*}
  &\widehat{\nabla}_X \phi Y+\mathcal{T}_X \omega Y = \phi\widehat{\nabla}_X Y+B\mathcal{T}_X Y    \\
  &\mathcal{T}_X \phi Y+\mathcal{H}\nabla_X \omega Y =
  \omega\widehat{\nabla}_X Y+C\mathcal{T}_X Y
\end{align*}
for $X,Y\in \Gamma(\ker F_*)$.
\item
\begin{align*}
  &\mathcal{V}\nabla_Z BW+\mathcal{A}_Z CW = \phi\mathcal{A}_Z W+B\mathcal{H}\nabla_Z W    \\
  &\mathcal{A}_Z BW+\mathcal{H}\nabla_Z CW = \omega\mathcal{A}_Z
  W+C\mathcal{H}\nabla_Z W
\end{align*}
for $Z,W\in \Gamma((\ker F_*)^{\perp})$.
\item
\begin{align*}
  &\widehat{\nabla}_X BZ+\mathcal{T}_X CZ = \phi\mathcal{T}_X Z+B\mathcal{H}\nabla_X Z    \\
  &\mathcal{T}_X BZ+\mathcal{H}\nabla_X CZ =
  \omega\mathcal{T}_X Z+C\mathcal{H}\nabla_X Z
\end{align*}
for $X\in \Gamma(\ker F_*)$ and $Z\in \Gamma((\ker F_*)^{\perp})$.
\end{enumerate}
\end{lemma}

\begin{theorem}
Let $F$ be a semi-slant submersion from an almost Hermitian
manifold $(M,g_M,J)$ onto a Riemannian manifold $(N,g_N)$. Then
the complex distribution $\mathcal{D}_1$ is integrable if and only
if we have
$$
\omega (\widehat{\nabla}_X Y-\widehat{\nabla}_Y X) =
C(\mathcal{T}_Y X-\mathcal{T}_X Y) \quad \text{for} \ X,Y\in
\Gamma(\mathcal{D}_1).
$$
\end{theorem}

\begin{proof}
For $X,Y\in \Gamma(\mathcal{D}_1)$ and $Z\in \Gamma((\ker
F_*)^{\perp})$, since $[X,Y]\in \Gamma(\ker F_*)$, we obtain
\begin{align*}
g_M (J[X,Y], Z)&= g_M (J(\nabla_X Y-\nabla_Y X), Z)    \\
      &= g_M (\phi \widehat{\nabla}_X Y+\omega \widehat{\nabla}_X Y+B\mathcal{T}_X Y+C\mathcal{T}_X Y
      -\phi\widehat{\nabla}_Y X-\omega\widehat{\nabla}_Y X    \\
      &\ \ \ -B\mathcal{T}_Y X-C\mathcal{T}_Y X, Z)    \\
      &= g_M (\omega \widehat{\nabla}_X Y+C\mathcal{T}_X Y-\omega\widehat{\nabla}_Y X-C\mathcal{T}_Y X, Z).
\end{align*}
Therefore, we have the result.
\end{proof}

Similarly, we get

\begin{theorem}
Let $F$ be a semi-slant submersion from an almost Hermitian
manifold $(M,g_M,J)$ onto a Riemannian manifold $(N,g_N)$. Then
the slant distribution $\mathcal{D}_2$ is integrable if and only
if we obtain
$$
P(\phi(\widehat{\nabla}_X Y-\widehat{\nabla}_Y X)+B(\mathcal{T}_X
Y-\mathcal{T}_Y X)) = 0 \quad \text{for} \ X,Y\in
\Gamma(\mathcal{D}_2).
$$
\end{theorem}

\begin{lemma}
Let $(M,g_M,J)$ be a K\"{a}hler manifold and $(N,g_N)$ a
Riemannian manifold. Let $F : (M,g_M,J) \mapsto (N,g_N)$ be a
semi-slant submersion. Then the slant distribution $\mathcal{D}_2$
is integrable if and only if we obtain
$$
P(\widehat{\nabla}_X \phi Y-\widehat{\nabla}_Y \phi
X+\mathcal{T}_X \omega Y-\mathcal{T}_Y \omega X) = 0 \quad
\text{for} \ X,Y\in \Gamma(\mathcal{D}_2).
$$
\end{lemma}

\begin{proof}
For $X,Y\in \Gamma(\mathcal{D}_2)$ and $Z\in
\Gamma(\mathcal{D}_1)$, since $[X,Y]\in \Gamma(\ker F_*)$, we have
\begin{align*}
g_M (J[X,Y], Z)&= g_M (\nabla_X JY-\nabla_Y JX, Z)    \\
      &= g_M (\widehat{\nabla}_X \phi Y+\mathcal{T}_X \phi Y+\mathcal{T}_X \omega Y
      +\mathcal{H}\nabla_X \omega Y-\widehat{\nabla}_Y\phi X-\mathcal{T}_Y \phi X    \\
      &\ \ \ -\mathcal{T}_Y \omega X-\mathcal{H}\nabla_Y \omega X, Z)    \\
      &= g_M (\widehat{\nabla}_X \phi Y+\mathcal{T}_X \omega Y-\widehat{\nabla}_Y\phi X
      -\mathcal{T}_Y \omega X, Z).
\end{align*}
Therefore, the result follows.
\end{proof}

In a similar way, we have

\begin{lemma}
Let $(M,g_M,J)$ be a K\"{a}hler manifold and $(N,g_N)$ a Riemannian
manifold. Let $F : (M,g_M,J) \mapsto (N,g_N)$ be a semi-slant
submersion. Then the complex distribution $\mathcal{D}_1$ is
integrable if and only if we get
$$
Q(\widehat{\nabla}_X \phi Y - \widehat{\nabla}_Y \phi X) = 0 \
\text{and} \ \mathcal{T}_X \phi Y = \mathcal{T}_Y \phi X \quad
\text{for} \ X,Y\in \Gamma(\mathcal{D}_1).
$$
\end{lemma}

Define an endomorphism $\widehat{F}$ of $\ker F_*$ by
$$
\widehat{F} := JP+\phi Q,
$$
where $(\nabla_X \widehat{F})Y := \widehat{\nabla}_X
\widehat{F}Y-\widehat{F}\widehat{\nabla}_X Y$ for $X,Y\in
\Gamma(\ker F_*)$. Then it is not difficult to get

\begin{lemma}
Let $F$ be a  semi-slant submersion from a K\"{a}hler manifold
$(M,g_M,J)$ onto a Riemannian manifold $(N,g_N)$. Then we have
$$
(\nabla_X \widehat{F})Y = \phi(\widehat{\nabla}_X
PY-\widehat{\nabla}_X Y)+B\mathcal{T}_X PY+\widehat{\nabla}_X \phi
QY \quad \text{for} \ X,Y\in \Gamma(\ker F_*).
$$
\end{lemma}

\begin{proposition}\label{prop:slant}
Let $F$ be a semi-slant submersion from an almost Hermitian manifold
$(M,g_M,J)$ onto a Riemannian manifold $(N,g_N)$. Then we obtain
$$
\phi^2 X = -\cos^2 \theta X \quad \text{for} \ X\in
\Gamma(\mathcal{D}_2),
$$
where $\theta$ denotes the semi-slant angle of $\mathcal{D}_2$.
\end{proposition}

\begin{proof}
Since
$$
\cos \theta = \frac{g_M (JX, \phi X)}{|JX|\cdot|\phi X|} =
\frac{-g_M (X, \phi^2 X)}{|X|\cdot|\phi X|}
$$
and $\displaystyle{ \cos \theta = \frac{|\phi X|}{|JX|} }$, we have
$$
\cos^2 \theta = -\frac{g_M (X, \phi^2 X)}{|X|^2} \quad \text{for} \
X\in \Gamma(\mathcal{D}_2).
$$
Hence,
$$
\phi^2 X = -\cos^2 \theta X \quad \text{for} \ X\in
\Gamma(\mathcal{D}_2).
$$
\end{proof}

\begin{remark}
In particular, we easily see that the converse of Proposition
\ref{prop:slant} is also true.
\end{remark}

Assume that the semi-slant angle $\theta$ is not equal to
$\displaystyle{\frac{\pi}{2}}$ and define an endomorphism
$\widehat{J}$ of $\ker F_*$ by
$$
\widehat{J} := JP+\frac{1}{\cos \theta}\phi Q.
$$
Then,
\begin{eqnarray} \label{compst}
{\widehat{J}}^2 = -id \quad \text{on} \ \ker F_*.
\end{eqnarray}

\begin{remark}
Let $F$ be a semi-slant submersion from an almost Hermitian manifold
$(M,g_M,J)$ onto a Riemannian manifold $(N,g_N)$. Assume that $\dim
M = 2m$, $\dim N = n$, and $\theta\in [0,\frac{\pi}{2})$. From
(\ref{compst}), we have
$$
\dim (\ker (F_*)_p)=2k \ \text{and} \ \dim ((\ker
(F_*)_p)^{\perp})=2m-2k \quad \text{for} \ p\in M,
$$
where $k$ is a non-negative integer.

Therefore, $n$ must be even.
\end{remark}

\begin{theorem}
Let $F$ be a semi-slant submersion from an almost Hermitian manifold
$(M,g_M,J)$ onto a Riemannian manifold $(N,g_N)$ with the semi-slant
angle $\theta\in [0,\frac{\pi}{2})$. Then $N$ is an even-dimensional
manifold.
\end{theorem}

\begin{proposition}
Let $F$ be a  semi-slant submersion from a K\"{a}hler manifold
$(M,g_M,J)$ onto a Riemannian manifold $(N,g_N)$. Then the
distribution $\ker F_*$ defines a totally geodesic foliation if and
only if
$$
\omega (\widehat{\nabla}_X \phi Y+\mathcal{T}_X \omega
Y)+C(\mathcal{T}_X \phi Y+\mathcal{H}\nabla_X \omega Y) = 0 \quad
\text{for} \ X,Y\in \Gamma(\ker F_*).
$$
\end{proposition}

\begin{proof}
For $X,Y\in \Gamma(\ker F_*)$,
\begin{align*}
\nabla_X Y&= -J\nabla_X JY= -J(\widehat{\nabla}_X \phi Y+\mathcal{T}_X \phi Y+\mathcal{T}_X \omega Y
             +\mathcal{H}\nabla_X \omega Y)   \\
          &= -(\phi \widehat{\nabla}_X \phi Y+\omega\widehat{\nabla}_X \phi Y+B\mathcal{T}_X \phi Y
             +C\mathcal{T}_X \phi Y+\phi \mathcal{T}_X \omega Y+\omega \mathcal{T}_X \omega Y    \\
          & \ \ \  +B\mathcal{H}\nabla_X \omega Y+C\mathcal{H}\nabla_X \omega
             Y).
\end{align*}
Thus,
$$
\nabla_X Y\in \Gamma(\ker F_*) \Leftrightarrow
\omega(\widehat{\nabla}_X \phi Y+\mathcal{T}_X \omega
Y)+C(\mathcal{T}_X \phi Y+\mathcal{H}\nabla_X \omega Y) =0.
$$
\end{proof}

Similarly, we have

\begin{proposition}
Let $F$ be a  semi-slant submersion from a K\"{a}hler manifold
$(M,g_M,J)$ onto a Riemannian manifold $(N,g_N)$. Then the
distribution $(\ker F_*)^{\perp}$ defines a totally geodesic
foliation if and only if
$$
\phi(\mathcal{V}{\nabla}_X BY+\mathcal{A}_X CY)+B(\mathcal{A}_X
BY+\mathcal{H}\nabla_X CY) = 0 \quad \text{for} \ X,Y\in
\Gamma((\ker F_*)^{\perp}).
$$
\end{proposition}

\begin{proposition}
Let $F$ be a  semi-slant submersion from a K\"{a}hler manifold
$(M,g_M,J)$ onto a Riemannian manifold $(N,g_N)$. Then the
distribution $\mathcal{D}_1$ defines a totally geodesic foliation if
and only if
$$
Q(\phi\widehat{\nabla}_X \phi Y+B\mathcal{T}_X \phi Y) = 0 \
\text{and} \ \omega\widehat{\nabla}_X \phi Y+C\mathcal{T}_X \phi Y =
0
$$
for $X,Y\in \Gamma(\mathcal{D}_1)$.
\end{proposition}

\begin{proof}
For $X,Y\in \Gamma(\mathcal{D}_1)$, we get
\begin{align*}
\nabla_X Y&= -J\nabla_X JY= -J(\widehat{\nabla}_X \phi Y+\mathcal{T}_X \phi Y)   \\
          &= -(\phi \widehat{\nabla}_X \phi Y+\omega\widehat{\nabla}_X \phi Y+B\mathcal{T}_X \phi Y
             +C\mathcal{T}_X \phi Y).
\end{align*}
Hence,
$$
\nabla_X Y\in \Gamma(\mathcal{D}_1) \Leftrightarrow Q(\phi
\widehat{\nabla}_X \phi Y+B\mathcal{T}_X \phi Y) = 0 \ \text{and} \
\omega\widehat{\nabla}_X \phi Y+C\mathcal{T}_X \phi Y = 0.
$$
\end{proof}

In a similar way, we obtain

\begin{proposition}
Let $F$ be a  semi-slant submersion from a K\"{a}hler manifold
$(M,g_M,J)$ onto a Riemannian manifold $(N,g_N)$. Then the
distribution $\mathcal{D}_2$ defines a totally geodesic foliation if
and only if
\begin{align*}
  &P(\phi(\widehat{\nabla}_X \phi Y+\mathcal{T}_X \omega Y)+B(\mathcal{T}_X \phi Y
    +\mathcal{H}\nabla_X \omega Y)) = 0     \\
  &\omega(\widehat{\nabla}_X \phi Y+\mathcal{T}_X \omega Y)+C(\mathcal{T}_X \phi Y
    +\mathcal{H}\nabla_X \omega Y) = 0
\end{align*}
for $X,Y\in \Gamma(\mathcal{D}_2)$.
\end{proposition}

\begin{theorem}
Let $F$ be a  semi-slant submersion from a K\"{a}hler manifold
$(M,g_M,J)$ onto a Riemannian manifold $(N,g_N)$. Then $F$ is a
totally geodesic map if and only if
\begin{align*}
  &\omega(\widehat{\nabla}_X \phi Y+\mathcal{T}_X \omega Y)+C(\mathcal{T}_X \phi Y
    +\mathcal{H}\nabla_X \omega Y) = 0   \\
  &\omega(\widehat{\nabla}_X BZ+\mathcal{T}_X CZ)+C(\mathcal{T}_X BZ
    +\mathcal{H}\nabla_X CZ) = 0
\end{align*}
for $X,Y\in \Gamma(\ker F_*)$ and $Z\in \Gamma((\ker F_*)^{\perp})$.
\end{theorem}

\begin{proof}
Since $F$ is a Riemannian submersion, we have
$$
(\nabla F_*)(Z_1,Z_2)=0 \quad \text{for} \ Z_1,Z_2\in \Gamma((\ker
F_*)^{\perp}).
$$
For $X,Y\in \Gamma(\ker F_*)$, we obtain
\begin{align*}
(\nabla F_*)(X,Y)&= -F_* (\nabla_X Y) = F_* (J\nabla_X (\phi Y+\omega Y))   \\
          &= F_* (\phi \widehat{\nabla}_X \phi Y+\omega\widehat{\nabla}_X \phi Y+B\mathcal{T}_X \phi Y
             +C\mathcal{T}_X \phi Y+\phi \mathcal{T}_X \omega Y+\omega \mathcal{T}_X \omega Y   \\
          & \ \ \  +B\mathcal{H}\nabla_X \omega Y+C\mathcal{H}\nabla_X \omega Y).
\end{align*}
Thus,
$$
(\nabla F_*)(X,Y) = 0 \Leftrightarrow \omega(\widehat{\nabla}_X \phi
Y+\mathcal{T}_X \omega Y)+C(\mathcal{T}_X \phi Y
    +\mathcal{H}\nabla_X \omega Y) = 0.
$$
For $X\in \Gamma(\ker F_*)$ and $Z\in \Gamma((\ker F_*)^{\perp})$,
 we get
\begin{align*}
(\nabla F_*)(X,Z)&= -F_* (\nabla_X Z) = F_* (J\nabla_X (BZ+CZ))   \\
          &= F_* (\phi \widehat{\nabla}_X BZ+\omega\widehat{\nabla}_X BZ+B\mathcal{T}_X BZ
             +C\mathcal{T}_X BZ+\phi \mathcal{T}_X CZ+\omega \mathcal{T}_X CZ   \\
          & \ \ \  +B\mathcal{H}\nabla_X CZ+C\mathcal{H}\nabla_X CZ).
\end{align*}
Hence,
$$
(\nabla F_*)(X,Z) = 0 \Leftrightarrow \omega(\widehat{\nabla}_X
BZ+\mathcal{T}_X CZ)+C(\mathcal{T}_X BZ
    +\mathcal{H}\nabla_X CZ) = 0.
$$
Since $(\nabla F_*)(X,Z)=(\nabla F_*)(Z,X)$, we get the result.
\end{proof}

Let $F$ be a  semi-slant submersion from a K\"{a}hler manifold
$(M,g_M,J)$ onto a Riemannian manifold $(N,g_N)$. Assume that $\mathcal{D}_1$ is integrable. Choose a local
orthonormal frame $\{ v_1, \cdots, v_l \}$ of $\mathcal{D}_2$ and a
local orthonormal frame $\{ e_1, \cdots, e_{2k} \}$ of
$\mathcal{D}_1$ such that $e_{2i}=Je_{2i-1}$ for $1\leq i\leq k$.
Since $\displaystyle{ F_*(\nabla_{Je_{2i-1}} Je_{2i-1}) =
-F_*(\nabla_{e_{2i-1}} e_{2i-1}) }$ for $1\leq i\leq k$, we have
$$
trace (\nabla F_*) = 0 \Leftrightarrow \sum_{j=1}^l F_*(\nabla_{v_j}
v_j) = 0.
$$

\begin{theorem}
Let $F$ be a  semi-slant submersion from a K\"{a}hler manifold
$(M,g_M,J)$ onto a Riemannian manifold $(N,g_N)$ such that $\mathcal{D}_1$ is integrable. Then $F$ is a
harmonic map if and only if
$$
trace (\nabla F_*) = 0 \quad \text{on} \ \mathcal{D}_2.
$$
\end{theorem}

Let $F : (M,g_M)\mapsto (N,g_N)$ be a Riemannian submersion. The map
$F$ is called a Riemannian submersion {\em with totally umbilical
fibers} if
\begin{eqnarray}\label{umbilic}
\mathcal{T}_X Y = g_M (X, Y)H \quad \text{for} \ X,Y\in \Gamma(\ker
F_*),
\end{eqnarray}
where $H$ is the mean curvature vector field of the fiber.

In a similar way with Lemma 4.2 of [18], we obtain

\begin{lemma}
Let $F$ be a semi-slant submersion with totally umbilical fibers
from a K\"{a}hler manifold $(M,g_M,J)$ onto a Riemannian manifold
$(N,g_N)$. Then we have
$$
H\in \Gamma(\omega \mathcal{D}_2).
$$
\end{lemma}

\begin{proof}
For $X,Y\in \Gamma(\mathcal{D}_1)$ and $W\in \Gamma(\mu)$, we get
$$
\mathcal{T}_X JY+\widehat{\nabla}_X JY = \nabla_X JY = J\nabla_X Y =
B\mathcal{T}_X Y+C\mathcal{T}_X Y+\phi\widehat{\nabla}_X
Y+\omega\widehat{\nabla}_X Y
$$
so that
$$
g_M (\mathcal{T}_X JY, W) = g_M (C\mathcal{T}_X Y, W).
$$
By (\ref{umbilic}), with a simple calculation we obtain
$$
g_M (X, JY) g_M (H, W) = -g_M (X, Y) g_M (H, JW).
$$
Interchanging the role of $X$ and $Y$, we get
$$
g_M (Y, JX) g_M (H, W) = -g_M (Y, X) g_M (H, JW)
$$
so that combining the above two equations, we have
$$
g_M (X, Y) g_M (H, JW) = 0
$$
which means $H\in \Gamma(\omega \mathcal{D}_2)$, since $J\mu = \mu$.
Therefore, we obtain the result.
\end{proof}

\begin{remark}
Let $F$ be a semi-slant submersion from a K\"{a}hler manifold $(M,
g_M, J)$ onto a Riemannian manifold $(N, g_N)$. Then there is a
distribution $\mathcal{D}_1\subset \ker F_*$ such that
$$
\ker F_* =\mathcal{D}_1\oplus \mathcal{D}_2, \
J(\mathcal{D}_1)=\mathcal{D}_1,
$$
and the angle $\theta=\theta(X)$ between $JX$ and the space
$(\mathcal{D}_2)_q$ is constant for nonzero $X\in (\mathcal{D}_2)_q$
and $q\in M$, where $\mathcal{D}_2$ is the orthogonal complement of
$\mathcal{D}_1$ in $\ker F_*$. Furthermore,
$$
\phi \mathcal{D}_2 \subset \mathcal{D}_2, \omega \mathcal{D}_2
\subset (\ker F_*)^{\perp}, (\ker
F_*)^{\perp}=\omega\mathcal{D}_2\oplus \mu,
$$
where $\mu$ is the orthogonal complement of $\omega \mathcal{D}_2$
in $(\ker F_*)^{\perp}$ and is invariant  under $J$. As we know, the
holomorphic sectional curvatures determine the Riemannian curvature
tensor in a K\"{a}hler manifold.

Given a plane $P$ being invariant by $J$ in $T_p M$, $p\in M$, there
is an orthonormal basis $\{ X, JX \}$ of $P$. Denote by $K(P)$,
$K_*(P)$, and $\widehat{K}(P)$ the sectional curvatures of the plane
$P$ in $M$, $N$, and the fiber $F^{-1}(F(p))$, respectively, where
$K_*(P)$ denotes the sectional curvature of the plane $P_* = < F_*X,
F_*JX > $ in $N$. Let $K(X\wedge Y)$ be the sectional curvature of
the plane spanned by the tangent vectors $X, Y\in T_p M$, $p\in M$.
Using both Corollary 1 of [14, p.465] and (1.27) of [7, p.12], we
obtain the following :

\begin{enumerate}

\item If $P\subset (\mathcal{D}_1)_p$, then with some computations
we have
$$
K(P)=\widehat{K}(P)+|\mathcal{T}_X X|^2-|\mathcal{T}_X JX|^2-g_M
(\mathcal{T}_X X, J[JX, X]).
$$

\item If $P\subset (\mathcal{D}_2\oplus \omega\mathcal{D}_2)_p$ with $X\in
(\mathcal{D}_2)_p$, then we get

\begin{align*}
K(P)&= \cos^2\theta \cdot K(X\wedge \phi X)+2g_M ((\nabla_{\phi X}
T)(X,X)-(\nabla_{X} T)(\phi X,X), \omega X)   \\
    & \ \ \  +\sin^2\theta \cdot K(X\wedge \omega X).
\end{align*}

\item If $P\subset (\mu)_p$, then we obtain
$$
K(P)= K_* (P)-3|\mathcal{V}J\nabla_X X|^2.
$$
\end{enumerate}
\end{remark}

\section{Examples}\label{exam}

\begin{example}
Let $F$ be a slant submersion from an almost Hermitian manifold
$(M,g_M,J)$ onto a Riemannian manifold $(N,g_N)$ \cite{S}. Then the
map $F$ is a semi-slant submersion with $\mathcal{D}_2 = \ker F_*$.
\end{example}

\begin{example}
Let $F$ be a semi-invariant submersion from an almost Hermitian
manifold $(M,g_M,J)$ onto a Riemannian manifold $(N,g_N)$ \cite{S2}.
Then the map $F$ is a semi-slant submersion with the semi-slant
angle $\theta=\frac{\pi}{2}$.
\end{example}

\begin{example}
Let $F$ be an almost h-slant submersion from a hyperk\"{a}hler
manifold $(M,g_M,I,J,K)$ onto a Riemannian manifold $(N,g_N)$ such
that $(I,J,K)$ is an almost h-slant basis \cite{P}. Then the map $F
: (M,g_M,R)\mapsto (N,g_N)$ is a semi-slant submersion with
$\mathcal{D}_2 = \ker F_*$ for $R\in \{ I,J,K \}$.
\end{example}

\begin{example}
Let $F$ be an almost h-semi-invariant submersion from a
hyperk\"{a}hler manifold $(M,g_M,I,J,K)$ onto a Riemannian manifold
$(N,g_N)$ such that $(I,J,K)$ is an almost h-semi-invariant basis
\cite{P2}. Then the map $F : (M,g_M,R)\mapsto (N,g_N)$ is a
semi-slant submersion with the semi-slant angle
$\theta=\frac{\pi}{2}$ for $R\in \{ I,J,K \}$.
\end{example}

\begin{example}
Define a map $F : \mathbb{R}^6 \mapsto \mathbb{R}^2$ by
$$
F(x_1,x_2,\cdots, x_6) = (x_3 \sin \alpha-x_5 \cos \alpha,x_6),
$$
where $\alpha\in (0,\frac{\pi}{2})$. Then the map $F$ is a
semi-slant submersion such that
$$
\mathcal{D}_1 = <\frac{\partial}{\partial x_1},
\frac{\partial}{\partial x_2}> \ \text{and} \ \mathcal{D}_2 =
<\frac{\partial}{\partial x_4}, \cos \alpha \frac{\partial}{\partial
x_3}+\sin \alpha \frac{\partial}{\partial x_5}>
$$
with the semi-slant angle $\theta=\alpha$.
\end{example}

\begin{example}
Define a map $F : \mathbb{R}^8 \mapsto \mathbb{R}^2$ by
$$
F(x_1,x_2,\cdots, x_8) = (\frac{x_5-x_8}{\sqrt{2}},x_6).
$$
Then the map $F$ is a semi-slant submersion such that
$$
\mathcal{D}_1 = <\frac{\partial}{\partial x_1},
\frac{\partial}{\partial x_2},\frac{\partial}{\partial
x_3},\frac{\partial}{\partial x_4}> \ \text{and} \ \mathcal{D}_2 =
<\frac{\partial}{\partial x_7}, \frac{\partial}{\partial
x_5}+\frac{\partial}{\partial x_8}>
$$
with the semi-slant angle $\theta=\frac{\pi}{4}$.
\end{example}

\begin{example}
Define a map $F : \mathbb{R}^{10} \mapsto \mathbb{R}^5$ by
$$
F(x_1,x_2,\cdots, x_{10}) =
(x_2,x_1,\frac{x_5+x_6}{\sqrt{2}},\frac{x_7+x_9}{\sqrt{2}},\frac{x_8+x_{10}}{\sqrt{2}}).
$$
Then the map $F$ is a semi-slant submersion such that
$$
\mathcal{D}_1 = <\frac{\partial}{\partial x_3},
\frac{\partial}{\partial x_4},-\frac{\partial}{\partial
x_7}+\frac{\partial}{\partial x_9},-\frac{\partial}{\partial
x_8}+\frac{\partial}{\partial x_{10}}> \ \text{and} \ \mathcal{D}_2
= <-\frac{\partial}{\partial x_5}+\frac{\partial}{\partial x_6}>
$$
with the semi-slant angle $\theta=\frac{\pi}{2}$.
\end{example}

\begin{example}
Define a map $F : \mathbb{R}^{10} \mapsto \mathbb{R}^4$ by
$$
F(x_1,x_2,\cdots, x_{10}) =
(\frac{x_3-x_5}{\sqrt{2}},x_6,\frac{x_7-x_9}{\sqrt{2}},x_8).
$$
Then the map $F$ is a semi-slant submersion such that
$$
\mathcal{D}_1 = <\frac{\partial}{\partial x_1},
\frac{\partial}{\partial x_2}> \ \text{and} \ \mathcal{D}_2 =
<\frac{\partial}{\partial x_3}+\frac{\partial}{\partial x_5},
\frac{\partial}{\partial x_7}+\frac{\partial}{\partial
x_9},\frac{\partial}{\partial x_4},\frac{\partial}{\partial x_{10}}>
$$
with the semi-slant angle $\theta=\frac{\pi}{4}$.
\end{example}

\begin{example}
Define a map $F : \mathbb{R}^8 \mapsto \mathbb{R}^4$ by
$$
F(x_1,x_2,\cdots, x_8) = (x_1,x_2,x_3 \cos \alpha-x_5 \sin
\alpha,x_4\sin \beta-x_6\cos \beta),
$$
where $\alpha$ and $\beta$ are constant. Then the map $F$ is a
semi-slant submersion such that
$$
\mathcal{D}_1 = <\frac{\partial}{\partial
x_7},\frac{\partial}{\partial x_8}> \ \text{and} \ \mathcal{D}_2 =
<\sin \alpha\frac{\partial}{\partial x_3}+\cos
\alpha\frac{\partial}{\partial x_5}, \cos
\beta\frac{\partial}{\partial x_4}+\sin
\beta\frac{\partial}{\partial x_6}>
$$
with the semi-slant angle $\theta$ with $\cos\theta=|\sin
(\alpha+\beta)|$.
\end{example}

\begin{example}
Let G be a slant submersion from an almost Hermitian manifold
$(M_1,g_{M_1},J_1)$ onto a Riemannian manifold $(N,g_N)$ with the
slant angle $\theta$ and $(M_2,g_{M_2},J_2)$ an almost Hermitian
manifold. Denote by $(M,g,J)$ the warped product of
$(M_1,g_{M_1},J_1)$ and $(M_2,g_{M_2},J_2)$ by a positive function
$f$ on $M_1$ [7], where $J=J_1\times J_2$. Define a map $F :
 (M,g,J)\mapsto (N,g_N)$ by
$$
F(x,y) = G(x) \quad \text{for} \ x\in M_1 \ \text{and} \ y\in M_2.
$$
Then the map $F$ is a semi-slant submersion such that
$$
\mathcal{D}_1 = TM_2 \ \text{and} \ \mathcal{D}_2 = \ker G_*
$$
with the semi-slant angle $\theta$.
\end{example}

\end{document}